\documentclass[12pt, reqno]{amsart}
\usepackage{graphicx,xcolor,cite,mathtools}
\usepackage{amssymb,amsmath,amsthm,mathrsfs,paralist,bm,esint,setspace}
\RequirePackage[colorlinks,citecolor=blue,urlcolor=blue]{hyperref}

\usepackage{color}

\newcommand{\C}{{\mathbb{C}}}
\newcommand{\N}{{\mathbb{N}}}
\newcommand{\R}{{\mathbb{R}}}
\newcommand{\Z}{{\mathbb{Z}}}
\newcommand{\E}{\mathrm{E}}

\renewcommand{\P}{\mathrm{P}}
\renewcommand{\d}{\mathrm{d}}

\newcommand{\e}{\mathrm{e}}
\newcommand{\Var}{\text{\rm Var}}
\newcommand{\lip}{\text{\rm Lip}}

\DeclareMathOperator{\Cov}{\text{\rm Cov}}

\DeclarePairedDelimiter\floor{\lfloor}{\rfloor}

\title[Fluctuation for SHE with H\"older coefficients]{Spatial fluctuation for stochastic heat equation with H\"older  coefficients}
\author{Carl Mueller and Fei Pu}
\date{}                                           
 \address[Carl Mueller]{Department of Mathematics, University of Rochester, Rochester, New
York 14627, USA\\
\textit{Email: carl.e.mueller@rochester.edu}
}

\address[Fei Pu]{Laboratory of Mathematics and Complex Systems,
School of Mathematical Sciences, Beijing Normal University, 100875, Beijing, China\\
\textit{Email: fei.pu@bnu.edu.cn}
}

\begin{document}
\newtheorem{stat}{Statement}[section]
\newtheorem{proposition}[stat]{Proposition}
\newtheorem*{prop}{Proposition}
\newtheorem{corollary}[stat]{Corollary}
\newtheorem{theorem}[stat]{Theorem}
\newtheorem{lemma}[stat]{Lemma}
\theoremstyle{definition}
\newtheorem{definition}[stat]{Definition}
\newtheorem*{cremark}{Remark}
\newtheorem{remark}[stat]{Remark}
\newtheorem*{OP}{Open Problem}
\newtheorem{example}[stat]{Example}
\newtheorem{nota}[stat]{Notation}
\numberwithin{equation}{section}
\maketitle

\begin{abstract}
           In this paper, we establish associativity, spatial ergodicity 
and a central limit theorem for certain nonnegative solutions to
           the stochastic heat equation $\partial_t u=\frac12\partial_x^2 u+ u^\gamma \xi$ with $\gamma\in (0, 1)$.
           When $\gamma=\frac12$, we derive a limit for the moment generating function of the spatial integral 
           and provide a lower bound on the spatial growth of the solution.            
\end{abstract}

\section{Introduction}

We study nonnegative solutions to the following stochastic 
heat equation.
\begin{align}\label{eq:SHE}
\begin{cases}
\partial_tu(t,x)=\frac12\partial_x^2u(t,x)+ u(t,x)^\gamma \xi(t,x), \quad t>0, x\in \R\\
u(0)\equiv1,
\end{cases}
\end{align}
with $\gamma\in (0, 1)$, where $\xi=\xi(t,x)$ denotes space-time 
white noise.  As usual, we regard \eqref{eq:SHE} as shorthand for the 
following mild equation.  
\begin{align}\label{eq:mild}
u(t,x)= 1+ \int_{[0,t]\times\R}p_{t-s}(x-y)u(s,y)^\gamma \xi(\d s\, \d y),
\end{align}
where $p_t(x)= (2\pi t)^{-1/2}\e^{-x^2/(2t)}$ is the heat kernel and the stochastic integral is in sense of Walsh \cite{Wal86}.  The 
existence of nonnegative weak solutions to \eqref{eq:SHE} was proved in 
\cite{MP92, Shi94} by tightness arguments.  When $\gamma=\frac12$, 
uniqueness in law among nonnegative solutions is established by the well-known exponential duality 
between $u(t,x)$ and nonnegative solutions $v(t,x)$ of the PDE
\begin{align*}
\partial_t v=\frac12 \partial_x^2v -\frac12 v^2.
\end{align*}
Indeed, the solution to \eqref{eq:SHE} with $\gamma=\frac12$ corresponds 
to the density of one-dimensional super Brownian motion (see Section III.4 
of \cite{Per02}).  For diffusion coefficient $|u(t,x)|^\gamma$ with $\frac34 <\gamma <1$, pathwise uniqueness among 
all solutions was proved in \cite{MP11}.


When $\frac12 <\gamma <1$ and for $|u(0,\cdot)|$ rapidly decreasing,  
Mytnik \cite{Myt98} extended exponential duality and hence proved 
uniqueness in law among nonnegative solutions.  Since our initial data is 
not rapidly decreasing, we cannot assume uniqueness in law for 
$\frac{1}{2}<\gamma\le\frac{3}{4}$.

To cover all of these possibilities, we define $\mathbf{U}_\gamma$ to be 
the set of nonnegative weak solutions to \eqref{eq:SHE}.  

We are interested in the spatial asymptotic behavior of the solution to \eqref{eq:SHE}.  For stochastic heat equation with diffusion coefficient
$\sigma(u)$, where $\sigma$ is a Lipchitz continuous function on $\R$, Huang, Nualart and Viitasaari \cite{HNV20} initiated the study of 
central limit theorem for the spatial average of the solution using Malliavin-Stein's method.  Fix $T>0$.  According to \cite[Theorem 1.2]{HNV20},  as $N\to\infty$,
          \begin{align*}
         \left\{ \frac{1}{\sqrt{N}}\int_0^{N}[u(t, x) -1] \, \d x\right\}_{t\in [0, T]}	\xrightarrow{C[0,T]} \left\{\int_0^{t}\sqrt{\E[\sigma(u(s, 0))^2]}\, \d {\rm B}_s\right\}_{t\in [0, T]}
          \end{align*}
where ${\rm B}$ denotes  Brownian motion and ``$\xrightarrow{C[0,T]}$'' denotes the convergence in law in the space of continuous functions $C[0,T]$. Since the work of Huang, Nualart and Viitasaari \cite{HNV20},  there has been a rapidly growing research on CLT for SPDEs with Lipschitz continuous diffusion coefficient; see, for instance, \cite{BNZ21,DNZ20, Ebi24, Ebi25, HNVZ20, NXZ22} and references therein.
Meanwhile, spatial ergodicity for stochastic heat equation Lipschitz continuous diffusion coefficient has been established by Chen et al. \cite{CKNP21} using Poincar\'e inequality.  See also \cite{NZ20} for spatial ergodicity of stochastic wave equations in dimensions $1, 2$ and $3$.

Li and Pu \cite{LP23} consider the square root diffusion coefficient. Using the Laplace functional of super Brownian motion, they show that 
as a process in time, 
the normalization of the spatial average of the solution to \eqref{eq:SHE}  converges in law to the standard Brownian motion in the space of continuous functions; see \cite[Theorem 1.1]{LP23}. However, the spatial ergodicity for the density of super Brownian motion is unknown.

We are aiming to establish the spatial ergodicity and CLT for the solution to \eqref{eq:SHE}.

\begin{theorem}\label{th:CLT}
          Suppose $\gamma=\frac{1}{2}$ or $\gamma\in(\frac{3}{4},1)$, and 
let $u\in\mathbf{U}_\gamma$.  Then the following assertions hold.  If 
$\gamma\in (0, \frac12)\cup(\frac{1}{2},\frac{3}{4}]$, then there exists 
$u\in\mathbf{U}_\gamma$ such that these assertions hold.  
\begin{enumerate}
\item For each $t>0$, we have that $u(t,\cdot)$ is spatially ergodic.
\item  Fix $T>0$. As $N\to\infty$, 
          \begin{align*}
         \left\{ \frac{1}{\sqrt{N}}\int_0^{N}[u(t, x) -1] \, \d x\right\}_{t\in [0, T]}	\xrightarrow{C[0,T]} \left\{\int_0^{t}\sqrt{\E[u(s, 0)^{2\gamma}]}\, \d {\rm B}_s\right\}_{t\in [0, T]}.
          \end{align*}
\end{enumerate}
\end{theorem}

In \cite{HNV20, CKNP21}, Malliavin calculus plays a crucial role in the 
study of spatial ergodicity and CLT for stochastic heat equation with 
Lipschitz continuous coefficient. For H\"older continuous coefficients, we 
cannot apply Malliavin calculus to study spatial limit theorems of the 
solution, at least not in a direct way.  In fact, when 
$\frac12\leq \gamma\leq \frac34$, we do not know if the solution is a 
functional of the underlying space-time white noise. Even when 
$\frac34<\gamma<1$, there exists a unique strong solution (see 
\cite[Theorem 1.3]{MP11}), it is not clear to us if the solution is 
Malliavin differentiable.  By the exponential duality, the Laplace 
functional of the solution to \eqref{eq:SHE} is available. In the case 
$\frac12<\gamma<1$, the dual process is a solution to an SPDE driven by a 
one-sided stable process (see \cite{Myt98}). However, we will not apply 
this approach to prove the CLT in Theorem \ref{th:CLT}. 

The main tool to prove Theorem \ref{th:CLT} is associativity.  We recall from Esary et al. \cite{EPW67} that
a random vector $X:=(X_1\,,\ldots,X_m)$ is said to be \emph{associated} if
	\begin{equation}\label{E:assoc}
		\Cov[h_1(X)\,,h_2(X)]\ge0,
	\end{equation}
for every pair of functions $h_1,h_2:\R^m\to\R$ that are nondecreasing in 
every coordinate and satisfy $h_1(X),h_2(X)\in L^2(\Omega)$. A random 
field $\Phi=\{\Phi(x)\}_{x\in\R^d}$ is \emph{associated} if 
$(\Phi(x_1)\,,\ldots,\Phi(x_m))$ is associated for every 
$x_1,\ldots,x_m\in\R^d$. We remark that an associated random vector is 
also said to satisfy the FKG inequalities; see Newman \cite{New80}.

\begin{theorem}\label{th:association}
          Suppose $\gamma=\frac{1}{2}$ or $\gamma\in(\frac{3}{4},1)$, and 
let $u\in\mathbf{U}_\gamma$.  Then the following conclusion holds.  If 
$\gamma\in (0, \frac12)\cup(\frac{1}{2},\frac{3}{4}]$, then there exists 
$u\in\mathbf{U}_\gamma$ such that the conclusion holds.  

       \textbf{Conclusion:}  $u$ is associated. 
\end{theorem}

Chen et al. \cite[Theorem A.4]{CKNP23} prove associativity for a class of 
stochastic heat equations with Lipschitz continuous diffusion 
coefficients.  Using  the comparison principle for the stochastic heat 
equation, they show that under certain assumptions on the initial data and 
diffusion coefficient, the Malliavin derivative of the solution is 
nonnegative. Then an appeal to Clark-Ocone formula shows that 
\eqref{E:assoc} holds and hence proves associativity. See the proof of 
\cite[Theorem A.4]{CKNP23} for details. As we mentioned before, it does 
not make sense to show that Malliavin derivative of the solution to 
\eqref{eq:SHE} is nonnegative. However, as in the proof of existence 
of weak solution to \eqref{eq:SHE} (see \cite[Theorem 1.1]{MP11} and 
\cite[Theorem 1.2]{MPS06}), we will approximate the H\"older coefficient 
by Lipschitz continuous functions and show that each approximating process 
is associated by  \cite[Theorem A.4]{CKNP23}. Thus, as a weak limit of 
these approximating processes,  a solution to \eqref{eq:SHE} is 
associated. The detailed proof of Theorem \ref{th:association} will be 
given in Section \ref{se:ass}. Then we will apply Theorem \ref{th:association} to
prove Theorem \ref{th:CLT} in Section \ref{sec:CLT}.

According to Theorem \ref{th:CLT}, the spatial integral of the solution to \eqref{eq:SHE} satisfies law of large numbers and CLT. 
It is natural to ask if it satisfies large deviation principle. In Section \ref{sec:LDP},  we will consider the moment generating 
function of the spatial integral of the solution to \eqref{eq:SHE} and derive a precise limit in the case $\gamma=\frac12$.
We can also apply associativity to study the spatial growth. In Section \ref{sec:max}, we consider $\gamma=\frac12$
and give a lower bound on the asymptotic maximum of the solution to \eqref{eq:SHE} (based on an assumption on the distribution of $u(t,0)$).

We conclude this section with a brief overview of the notation used throughout the paper.
For every $Y\in L^k(\Omega)$ with $k\in[1,\infty)$, we write $\|Y\|_k=(\E[|Y|^k])^{1/k}$.  Throughout we write ``$g_1(x)\lesssim g_2(x)$ for all $x\in X$'' when
there exists a real number $L$ such that $g_1(x)\le Lg_2(x)$ for all $x\in X$.
Define
\begin{align*}
\|f\|_\lambda = \sup_{x\in \R}|f(x)|\e^{-\lambda |x|}
\end{align*}
and set 
$C_{tem}(\R)= \{f\in C(\R): \|f\|_\lambda <\infty\,\, \text{for any $ \lambda>0$}\}$. 
Finally, symbols such as $C,C_T,c,c_n$, without any explicit dependence on 
$\R$ or other spaces, will denote constants which may change from 
line to line.

\section{Associativity}\label{se:ass}

We prove associativity in this section.

 \begin{proof}[Proof of Theorem \ref{th:association}]
          We will modify the arguments in the proof of \cite[Theorem 1.1]{MP11} to approximate the solution to \eqref{eq:SHE}.
          Choose an even function $\psi_n \in C_c^{\infty}(\R)$ so that 
$0\leq \psi_n\leq 1$, $\|\psi_n'\|_\infty \leq 1$, $\psi_n(x)=1$ if 
$|x|\leq n$ and $\psi_n(x)=0$ if $|x|\geq n+2$. For $n\geq 1$, define the 
function
\begin{align*}
\sigma_n(x)= \psi_n(x) \left(|x|^\gamma1_{\{|x|>1/n\}} + n^{1-\gamma}|x|1_{\{|x|\leq 1/n\}}\right), \quad x\in \R.
\end{align*}
Then for each $n\geq1$, the function $\sigma_n$ satisfies 
\begin{align}
|\sigma_n(x)|& \leq C\cdot(1+|x|), \quad\text{for all $x\in \R$} \nonumber\\
|\sigma_n(x)-\sigma_n(y)| &\leq c_n|x-y|, \quad \text{for all $x, y\in \R$} \label{eq:lip}\\
\sigma_n(x) &\geq 0, \quad\text{for all $x\in \R$} \label{eq:ass}\\
\sigma_n(0)&=0. \label{eq:non}
\end{align}
In addition,
\begin{align*}
          \left|\sigma_n(x)- |x|^\gamma\right| &\leq |1-\psi_n(x)| \left(|x|^\gamma1_{\{|x|>1/n\}} + n^{1-\gamma}|x|1_{\{|x|\leq 1/n\}}\right)\\
          &\quad + 1_{\{|x|\leq 1/n\}} \left||x|^\gamma- n^{1-\gamma}|x|\right|\\
          &\leq  |1-\psi_n(x)| (2+|x|) + n^{-\gamma}\left(\gamma^{\gamma/(1-\gamma)}- \gamma^{1/(1-\gamma)}\right).
\end{align*}
Thus, $\sigma_n$ converges to $|x|^\gamma$ uniformly on compact sets as $n\to \infty$.

Condition \eqref{eq:lip} ensures that 
for each $n\geq 1$, there exists a unique solution (denoted by $u_n$) to the following equation
 \begin{align}\label{eq:sigma_n}
\begin{cases}
\partial_tu_n(t,x)=\frac12\partial_x^2u_n(t,x)+ \sigma_n(u_n(t,x)) \xi(t,x),\\
u_n(0)\equiv1.
\end{cases}
\end{align}
Because of \eqref{eq:non}, the solution $u_n$ is nonnegative (see \cite{Mue91, Shi94}). 
Moreover, by  \eqref{eq:ass} and \cite[Theorem A.4]{CKNP23}, $u_n$ is  associated.  
Furthermore, from the proof of 
  \cite[Theorem 1.1]{MP11},  there exists a subsequence $(n_k)_{k=1}^\infty$ such that as $k\to\infty$, $u_{n_k}$ converges a.s. to $u$ in $C(\R_+, C_{tem}(\R))$ on some probability space, where $u$ is a weak solution to the following stochastic heat equation
   \begin{align*}
\begin{cases}
\partial_tu(t,x)=\frac12\partial_x^2u(t,x)+ |u(t,x)|^\gamma \xi(t,x), \\
u(0)\equiv1.
\end{cases}
\end{align*}
Clearly,  $u$ is nonnegative, which implies that $u$ is a weak solution to \eqref{eq:SHE}. 
Moreover, the almost sure convergence of $u_{n_k}$ to $u$ in\\ 
$C(\R_+, C_{tem}(\R))$ implies that the finite-dimensional 
distributions of $u_{n_k}$ converge to those of $u$ as $k\to\infty$.  
Since convergence of finite-dimensional distributions preserves 
associativity (see \cite[(P$_5$)]{EPW67}), it follows that the process $u$ 
is associated.  If $\gamma=\frac12$ or $\gamma\in (\frac34, 1)$, then \eqref{eq:SHE} has a unique solution in 
law and it follows that any $u\in \mathbf{U}_{\gamma}$ is associated. If $\gamma\in (0, \frac12)\cup(\frac12, \frac34]$, we conclude that 
there exists  an associated $u$ in $\mathbf{U}_\gamma$.
The proof of Theorem \ref{th:association} is complete. 
\end{proof}

\begin{remark}\label{rem:sta}
(1) From the proof of Theorem \ref{th:association}, we see that 
associativity holds for the solution to \eqref{eq:SHE} with some general 
nonnegative initial condition.  As a consequence of associativity, we have 
for $(t_1, x_1), \ldots, (t_m, x_m)\in \R_+\times\R$ and for Lipschitz 
continuous functions $f, g: \R^m\mapsto \R$, 
\begin{align}\label{Newman2}
         & \left|\Cov(f(u(t_1,x_1), \ldots, u(t_m,x_m))\,, g(u(t_1,x_1), \ldots, u(t_m,x_m)))\right| \nonumber\\
          &\qquad\qquad\qquad \leq \sum_{j=1}^m\sum_{\ell=1}^m
           \lip_j(f)\lip_\ell(g) \Cov(u(t_j,x_j)\,, u(t_\ell, x_\ell)),
\end{align}
where $\lip_j(f)$ denotes the Lipschitz constant with respect to the $j$-th component; see Bulinski and Shabanovich \cite{BuS98} and  \cite[Theorem 6.2.6]{PR12}.

\smallskip
(2) Fix $t_1, \ldots, t_k\in \R_+$. According to \cite[Lemma 18]{Dal99}, for each $n\geq1$, the process $\{u_n(t_1, x), \ldots, u_n(t_k,x)\}_{x\in \R}$ is stationary, where $u_n$ is the solution to \eqref{eq:sigma_n}.  As a consequence, the process 
$\{u(t_1, x), \ldots, u(t_k,x)\}_{x\in \R}$ is stationary, where $u$ is the weak limit of $u_{n_k}$ and is a weak solution to \eqref{eq:SHE}.

\end{remark}

\section{Spatial ergodicity and CLT}\label{sec:CLT}

In this section, we prove Theorem \ref{th:CLT} using Theorem \ref{th:association}. We  start by giving an estimate on the spatial covariance of the solution to \eqref{eq:SHE}.

\begin{lemma}\label{lem:cov}
          Fix $T>0$. There exists a constant $C_T>0$ such that for all $t\in [0, T]$ and $x\in \R$
          \begin{align*}
          \Cov(u(t,x)\,, u(t,0)) \leq C_T\e^{-\frac{x^2}{4T}}.
          \end{align*}
\end{lemma}
\begin{proof}
  By the moment estimates in \cite[Theorem 1.6]{CX23},  for $k\geq2$, there exists a constant $c_k>0$ such that
\begin{align}\label{eq:moment}
\|u(t,x)\|_k^2 \leq c_k t^{\frac{1}{2(1-\gamma)}}
\end{align}
all for $t>0$.  By stationarity and Ito's isometry, we have
\begin{align*}
          \Cov(u(t, x)\,, u(t,0))& = \int_0^t\int_\R p_{t-s}(x-y)p_{t-s}(y)\E\left[u(s,y)^{2\gamma}\right]\d y\d s\\
          &=\int_0^t \E\left[u(s,0)^{2\gamma}\right] p_{2(t-s)}(x)\d s\\
          &\le c_{T} \int_0^Tp_{2s}(x)\d s \leq c_T \sqrt{\frac{T}{\pi}}\e^{-\frac{x^2}{4T}}.
\end{align*}
The proof is complete.
\end{proof}

\begin{proof}[Proof of Theorem \ref{th:CLT}: spatial ergodicity] 
          Chen et al. \cite[Lemma 7.2]{CKNP21} provides a general criterion on the ergodicity of a stationary process; see also \cite[Lemma 4.2]{BaZ24}. Indeed, we can apply Lemma \ref{lem:cov} and the inequality \eqref{Newman2} and follow along 
          the same lines as in the proof of \cite[Theorem 1.1]{Pu25} to obtain that for each $t>0$, $\{u(t,x):x \in \R\}$ is ergodic.
\end{proof}

In order to prove the CLT in Theorem \ref{th:CLT}, we first establish the tightness. For $t>0$ and $N\geq1$, denote 
\begin{align}\label{eq:int}
S_{N,t} = \int_0^N u(t,x) \, \d x. 
\end{align}

\begin{proposition} \label{prop:3.2}
          For $T>0$ and $k\geq 2$,  there exists a constant $C_{T, k}>0$ such that for all $s, t \in [0, T]$ and $N\geq1$
          \begin{align}\label{eq:increment}
          \E \left[|S_{N,t}- S_{N,s}|^k\right]\leq C_{T,k}N^{k/2}|t-s|^{k/2}.
          \end{align}
          Moreover, for $t>0$,
          \begin{align}\label{eq:var}
          \lim_{N\to\infty}\frac{\E[(S_{N,t}-N)^2]}{N}=  \int_0^t \E\left[u(r,0)^{2\gamma}\right]\d r. 
          \end{align}
\end{proposition}
\begin{proof}
          The proof is similar to that of \cite[Proposition 4.1]{HNV20}. We write
          \begin{align*}
          S_{N,t}-S_{N,s}&= \int_0^N [u(t,x)-u(s,x)]\d x\\
          &=\int_0^T\int_\R \int_0^N\left[1_{\{r<t\}}p_{t-r}(x-z)-1_{\{r<s\}}p_{s-r}(x-z)\right]\d x \\
          &\hspace{6cm} \times u(r,z)^\gamma \xi(\d r\, \d z).
          \end{align*}
          By Burkholder's inequality and Minkowski's inequatlity, 
          \begin{align*}
          &\|S_{N,t}-S_{N,s}\|_k^2 \\
          & \quad \lesssim \int_0^T\int_\R \left[\int_0^N\left[1_{\{r<t\}}p_{t-r}(x-z)-1_{\{r<s\}}p_{s-r}(x-z)\right]\d x\right]^2 \\
          &\hspace{8cm}
          \times\|u(r,0)^{\gamma}\|_k^2  \d z\d r\\
          & \quad \lesssim \int_0^T\int_\R \left[\int_0^N\left[1_{\{r<t\}}p_{t-r}(x-z)-1_{\{r<s\}}p_{s-r}(x-z)\right]\d x\right]^2\d z \d r,
          \end{align*}
          where the second inequality holds by \eqref{eq:moment}. Then, we apply \cite[(4.1)]{HNV20} to derive 
          \eqref{eq:increment}.
          
          Finally, we can perform the same calculation as in the proof of 
\cite[Proposition 3.1]{HNV20} to obtain \eqref{eq:var}.  
This finishes the proof of Proposition \ref{prop:3.2}
\end{proof}

We are now ready to prove the CLT in Theorem \ref{eq:SHE}.

\begin{proof}[Proof of Theorem \ref{eq:SHE}: CLT]
           Recall the spatial integral $S_{N,t}$ defined in \eqref{eq:int}.
           In light of \eqref{eq:increment}, we need show that the finite-dimensional distributions of 
           $\{\frac1{\sqrt{N}}(S_{N,t}-N)\}_{t\in [0, T]}$ converge to those of a centered Gaussian process, denoted by $\{\mathcal{G}_t\}_{t\in [0, T]}$
           with covariance 
           \begin{align*}
             \Cov(\mathcal{G}_{t_1}\,, \mathcal{G}_{t_2}) = \int_0^{t_1\wedge t_2} \E\left[u(s,0)^{2\gamma}\right] \d s.
           \end{align*}
           Fix $t_1, \ldots, t_k\in [0, T]$.  By Cram\'er-Wold theorem, we need prove that for all $a_1, \ldots, a_k\in \R$, 
           as $N\to\infty$,
           \begin{align}\label{eq:fdd}
           a_1 \frac{S_{N,t_1}-N}{\sqrt{N}}+ \ldots + a_k \frac{{S_{N,t_k}-N}}{\sqrt{N}}\xrightarrow{\rm d} a_1\mathcal{G}_{t_1}+
           \ldots + a_k\mathcal{G}_{t_k},
           \end{align}
           where the symbol $\xrightarrow{\rm d}$ refers to convergence in distribution. 
           Without loss of generality, we assume $N$ is an integer.
           
           We first show that \eqref{eq:fdd} holds for all positive $a_1, \ldots, a_k$. Fix $a_1$, $\ldots, a_k\in \R_+$.
            We write
           \begin{align*}
           \sum_{\ell=1}^k a_\ell (S_{N, t_\ell}-N)= \sum_{j=1}^NX_j,
           \end{align*}
           where
           \begin{align*}
           X_j= \int_{j-1}^j \sum_{\ell=1}^k (a_\ell u(t_\ell,x)-1) \d x, \quad j\in \Z.
           \end{align*}
           According to Remark \ref{rem:sta} (2), we see that 
$\{X_j\}_{j\in \Z}$ is a stationary sequence. 
           Moreover, we can approximate the Riemann integral that defines $X_j$ by Riemann sum and use 
           Theorem \ref{th:association} and the fact that 
$a_1, \ldots, a_k$ are nonnegative to conclude that
           the sequence $\{X_j\}_{j\in \Z}$ is associated. Furthermore, we have
           \begin{align*}
           \sigma^2&:=\sum_{j\in \Z} \Cov(X_0\,, X_j)\\
           &= \sum_{j\in \Z} \int_{-1}^0\int_{j-1}^j  
           \Cov(\sum_{\ell=1}^ka_\ell u(t_\ell, x)\,, \sum_{m=1}^ka_m u(t_m, y))\d y\d x\\
           &=\sum_{\ell=1}^k \sum_{m=1}^ka_\ell a_m \int_{-1}^0\int_\R
           \Cov( u(t_\ell, x)\,,  u(t_m, y))\d y\d x\\
           &=\sum_{\ell=1}^k \sum_{m=1}^ka_\ell a_m\int_\R
           \Cov( u(t_\ell, x)\,, u(t_m, 0))\d x\\
           &=\sum_{\ell=1}^k \sum_{m=1}^ka_\ell a_m \Cov(\mathcal{G}_{t_\ell}\,, \mathcal{G}_{t_m})<\infty. 
           \end{align*}
           Therefore, by \cite[Theorem 2]{New80}, we have as $N\to\infty$, 
           \begin{align*}
           \frac{X_1+\ldots +X_k}{\sqrt{N}} \xrightarrow{\rm d} {\rm N}(0, \sigma^2), 
           \end{align*}
           which implies that \eqref{eq:fdd} holds for all $a_1, \ldots, a_k\in \R_+$.

           We proceed to prove that \eqref{eq:fdd} holds for all $a_1, \ldots, a_k\in \R$. 
           Fix $a_1, \ldots, a_k\in \R$. By \eqref{eq:var}, the sequence of random vectors 
           $(\frac{S_{N,t_1}-N}{\sqrt{N}}$, $\ldots, \frac{{S_{N,t_k}-N}}{\sqrt{N}})_{N\in\N}$ is tight. 
           Thus, for any sequence $\{N_j\}_{j\in\N}$ such that $N_j\to \infty$ as $j\to\infty$, there exists a subsequence $\{N_j'\}$ such that 
           as $j\to\infty$,
           \begin{align*}
           \left(\frac{S_{N'_{j},t_1}-N_j'}{\sqrt{N'_{j}}}, \ldots, \frac{S_{N'_{j},t_k}-N_j'}{\sqrt{N'_{j}}}\right) \xrightarrow{\rm d} (G_1, \ldots, G_k)
           \end{align*}
           for some random vector $ (G_1, \ldots, G_k)$. 
           Since \eqref{eq:fdd} holds for all $a_1$, $\ldots, a_k\in \R_+$, we see that for all $a_1, \ldots, a_k\in \R_+$, the linear 
           combination $a_1G_1+\ldots + a_kG_k$ is a centered Gaussian random variable. Hence, by Lemma \ref{Gauss-linear}
           below, it follows that $(G_1, \ldots, G_k)$ is a centered Gaussian vector. Clearly, $(G_1, \ldots, G_k)$ has the same
           covariance matrix as $(\mathcal{G}_{t_1}, \ldots, \mathcal{G}_{t_k})$, which is independent of the choice of the sequence
           $\{N_k\}$. Therefore, we have as $N\to\infty$,
           \begin{align*}
            \left(\frac{S_{N,t_1}-N}{\sqrt{N}}, \ldots, \frac{S_{N,t_k}-N}{\sqrt{N}}\right) \xrightarrow{\rm d}
             (\mathcal{G}_{t_1}, \ldots, \mathcal{G}_{t_k}).
           \end{align*}
           
           The proof of Theorem \ref{th:CLT} is complete.
\end{proof}

\section{Large deviations}\label{sec:LDP}

Recall $S_{N,t}$ in \eqref{eq:int}. Fix $t>0$. If $N$ is an integer, we can write
\begin{align*}
S_{N,t} = \sum_{j=1}^{N}X_{j,t}, \quad \text{with}\quad X_{j,t} = \int_{j-1}^j u(t,x)\, \d x. 
\end{align*}
By Theorem \ref{th:CLT}, $S_{N,t}$ satisfies law of large numbers and CLT. A further question is that if $S_{N,t}$ satisfies large deviation principle. We refer to \cite{Ebi25b} for large deviations for  spatial average of stochastic heat and wave equations with bounded and Lipschitz continuous diffusion coefficients.  In principle, large deviation for a sum of random variables $S_{N,t}$ is related to the limit
\begin{align}\label{eq:LDP}
\lim_{N\to\infty} \frac1N\log \E\left[\e^{\theta S_{N,t}}\right], \quad \theta \in \R.
\end{align}
Since the sequence $\{X_{j,t}\}_{j\in \Z}$ is stationary and associated,  the limit in \eqref{eq:LDP} exists for all $\theta\in \R$ (see the proof of \cite[Theorem 3.16]{Oli12}). When $\gamma=\frac12$, we are able to compute the precise limit in \eqref{eq:LDP}
for negative $\theta$.

\begin{theorem}\label{th:LDP}
          Assume $\gamma=\frac12$. Fix $t>0$. For $\lambda>0$, it holds that
          \begin{align}\label{eq:LDP2}
           \lim_{N\to\infty} \frac1N\log \E\left[\e^{-\lambda S_{N,t}}\right]= -\frac{2\lambda}{\lambda t+2}.
          \end{align}
\end{theorem}
\begin{proof}
           When $\gamma=\frac12$, the solution to \eqref{eq:SHE} corresponds to the density of super Brownian motion. By 
           the exponential duality, we write for $\lambda>0$
\begin{align*}
\E\left[\e^{-\lambda S_{N,t}}\right] =\E\left[\e^{-\lambda \langle u(t, \cdot)\,, 
\bm{1}_{[0, N]}\rangle}\right] = \e^{-\langle 1, v(t, \cdot)\rangle},
\end{align*}
where $v$ solves the heat equation
\begin{align}\label{eq:HE1}
\begin{cases}
\partial_tv= \frac12 \partial_x^2 v -\frac12 v^2,\\
v(0, x)= \lambda \bm{1}_{[0, N]}(x).
\end{cases}
\end{align}
Thus, it remains to verify 
\begin{align}\label{eq:LDPlim}
\lim_{N\to\infty} \frac{\int_\R v(t,x)\, \d x}{N} = \frac{2\lambda}{\lambda t+2}.
\end{align}

The solution to \eqref{eq:HE1} is nonnegative and satisfies the following integral equation
\begin{align*}
v(t,x)= \int_\R \lambda \bm{1}_{[0, N]}(y) p_t(x-y)\d y -\frac12 \int_0^t\int_\R p_{t-s}(x-y)v(s, y)^2\d y\d s
\end{align*}
and hence we have 
\begin{align}\label{eq:vbound}
v(t,x ) \leq  \int_\R \lambda \bm{1}_{[0, N]}(y) p_t(x-y)\d y, \quad t>0, x\in\R. 
\end{align}

Let $w$ be the solution to 
\begin{align}\label{eq:HE2}
\begin{cases}
\partial_tw= \frac12 \partial_x^2 w -\frac12 w^2,\\
w(0, x)= \lambda.
\end{cases}
\end{align}
Note that $w$ depends only on $t$, so will write $w(t)$. Thus, $w$ solves an ODE with an explicit formula given by
\begin{align*}
w(t)= \frac{2\lambda}{\lambda t+2}, \quad t\geq0.
\end{align*}
Let 
\begin{align*}
h(t,x)= w(t,x)-v(t,x).
\end{align*}
Because the initial value of $w$ is larger than that of $v$, by the comparison principle, $h$ is nonnegative (and bounded above by $\lambda$). In addition,  we see that $h$ satisfies
\begin{align*}
\partial_t h &=\partial_t w-\partial_t v
=\frac12\partial_x^2 w-\frac12w^2-\frac12\partial_x^2v+\frac12v^2\\
&=\frac12\partial_x^2h-\frac12(w+v)h \leq \frac12\partial_x^2h.
\end{align*}
A supersolution for $h$ is $\bar{h}$, where
\begin{align}\label{eq:h}
h\leq \bar{h},
\end{align}
and $\bar{h}$ solves
\begin{align}\label{eq:HE3}
\begin{cases}
\partial_t \bar{h}=\frac12\partial_x^2\bar{h},\\
\bar{h}(0, x)= \lambda \bm{1}_{[0, N]^c}(x).
\end{cases}
\end{align}
The solution to \eqref{eq:HE3} is given by
\begin{align*}
\bar{h}(t,x)= \int_\R \lambda \bm{1}_{[0, N]^c}(y)p_t(x-y)\d y.
\end{align*}

In order to prove \eqref{eq:LDPlim},
we use \eqref{eq:vbound} and \eqref{eq:h} to see that
\begin{align*}
&\left|\frac1N\int_\R v(t,x)\d x -w(t)\right| 
\leq \frac1N\left[\int_{[0, N]^c} v(t,x)\d x  + \int_0^N \bar{h}(t,x)\d x\right]\\
&\qquad\leq  \frac\lambda N\Big[\int_{[0, N]^c}\left(  \int_\R \bm{1}_{[0, N]}(y) p_t(x-y)\d y\right)\d x  \\
&\qquad\qquad\qquad\qquad \qquad\qquad
+ \int_0^N \left( \int_\R \bm{1}_{[0, N]^c}(y)p_t(x-y)\d y\right)\d x\Big]\\
&\qquad =\frac{2\lambda} N\int_{[0, N]^c}\left(  \int_\R \bm{1}_{[0, N]}(y) p_t(x-y)\d y\right)\d x.
\end{align*}
We only need to show that
\begin{align}\label{eq:lim2}
\lim_{N\to\infty} \frac1N\int_{-\infty}^0 \left(\int_0^N p_t(x-y)\d y\right) \d x =0
\end{align}
and
\begin{align}\label{eq:lim}
  \lim_{N\to\infty} \frac1N\int^{\infty}_N \left(\int_0^N p_t(x-y)\d y\right) \d x =0
\end{align}
For the integral in \eqref{eq:lim2}, we write
\begin{align*}
 \frac1N\int_{-\infty}^0 \left(\int_0^N p_t(x-y)\d y\right) \d x= \frac1N \int_0^N \left(\int_{-\infty}^{-y}p_t(z)\d z\right)\d y.
\end{align*}
By L'H\^opital's rule, 
\begin{align*}
\lim_{N\to\infty} \frac1N \int_0^N \left(\int_{-\infty}^{-y}p_t(z)\d z\right)\d y = \lim_{N\to\infty} \int_{-\infty}^{-N}p_t(z)\d z =0.
\end{align*}
For the  integral in \eqref{eq:lim}, we write
\begin{align*}
&\frac1N\int^{\infty}_N \left(\int_0^N p_t(x-y)\d y\right) \d x\\
& \quad= \frac1N\int^{\infty}_{2N} \left(\int_0^N p_t(x-y)\d y\right) \d x
+ \frac1N\int^{2N}_N \left(\int_0^N p_t(x-y)\d y\right) \d x.
\end{align*}
It is clear that
\begin{equation*} 
\begin{split}
 \frac1N\int^{\infty}_{2N} \left(\int_0^N p_t(x-y)\d y\right) \d x &\leq \frac1N \int^{\infty}_{2N} \left(\int_0^N p_t(x/2)\d y\right) \d x\\
 &=\int_{2N}^{\infty}p_t(x/2)\d x \to 0, \quad \text{as $N\to\infty$.}
\end{split}
\end{equation*}
To justify the first line, we briefly explain why $|x-y|\ge x/2$ for $x,y$ in 
the domain of integration.  
Indeed, for such $x,y$ we have $x\ge2N$ and $y\le N$.  Therefore $y\le x/2$ 
and so $|x-y|=x-y\ge x-(x/2)=x/2$.  
Moreover, by a change of variable, 
\begin{align*}
\frac1N\int^{2N}_N \left(\int_0^N p_t(x-y)\d y\right) \d x&=\int_1^2\d x \int_0^1 \d y \frac{N}{\sqrt{2\pi t}} \e^{-\frac{N^2(x-y)^2}{2t}}\\
&=\int_1^2\d x \int_0^1 \d y \, \frac1{x-y}\frac{N(x-y)}{\sqrt{2\pi t}} \e^{-\frac{N^2(x-y)^2}{2t}}.
\end{align*}
Since 
\begin{align*}
\sup_{N\geq1, x, y\in \R} \left[N|x-y|\e^{-\frac{N^2|x-y|^2}{2t}}\right]<\infty \quad \text{and}\quad \int_1^2\d x \int_0^1 \d y \, \frac1{x-y}<\infty,
\end{align*}
by the dominated convergence theorem, 
\begin{align*}
\lim_{N\to\infty}\frac1N\int^{2N}_N \left(\int_0^N p_t(x-y)\d y\right) \d x=0.
\end{align*}
Therefore,  \eqref{eq:lim} holds and it proves \eqref{eq:LDPlim}. 

The proof of Theorem \ref{th:LDP} is complete.
\end{proof}

Based on Theorem \ref{th:LDP}, we can derive a upper bound on the lower tail probability of $\frac{S_{N,t}}{N}$.
\begin{corollary}
          Assume $\gamma=\frac12$. Fix $t>0$. Then for all $a\in (0, 1)$
          \begin{align}\label{eq:lower}
          \limsup_{N\to\infty} \frac1 N \log \P\left\{\frac{S_{N,t}}{N}<a\right\}\leq -\frac{2(1-\sqrt{a})^2}{t}.
          \end{align}
\end{corollary}
\begin{proof}
          For $\lambda >0$, by Markov's inequality,
          \begin{align*}
           \log \P\left\{\frac{S_{N,t}}{N}<a\right\} &=\log \P\left\{\e^{-\lambda S_{N,t}} >\e^{-\lambda Na}\right\}\\
           &\leq \lambda N a +\log  \E\left[\e^{-\lambda S_{N,t}}\right].
          \end{align*}
          Hence, by Theorem \ref{th:LDP}, for all $\lambda>0$ and $a\in (0, 1)$,
          \begin{align*}
          \limsup_{N\to\infty} \frac1 N \log \P\left\{\frac{S_{N,t}}{N}<a\right\} \leq \lambda a -\frac{2\lambda}{\lambda t+2}.
          \end{align*}
          We minimize the right-hand side of the above line over $\lambda >0$ to obtain \eqref{eq:lower}. 
\end{proof}

\section{Discussion on spatial growth}\label{sec:max}

Throughout this section, we assume $\gamma =\frac12$ and fix $t>0$.
Note that Hu et al. \cite[Proposition 1.4]{HWXZ24} have obtained the asymptotic of the upper tail of the solution to \eqref{eq:SHE}, namely, there exist positive constants $c_1, c_2$ (independent of $t$) and $z_0>0$ such that
\begin{align}\label{eq:tail}
\e^{-c_1t^{-1/2}z} \leq  \P\{u(t, 0)>z\} \leq \e^{-c_2t^{-1/2}z}  \quad \text{for all $z\geq z_0$}.
\end{align}
The above estimate on tail probability together with Kolmogorov continuity theorem implies that there exists a deterministic constant $c>0$ such that a.s.
\begin{align*}
\limsup_{N\to\infty} \frac{\max_{0\leq x\leq N}u(t,x)}{\log N} \leq c\, t^{1/2},
\end{align*}
(see \cite[Theorem 1.11]{CX23}).

We would like to get a lower bound on the asymptotic maximum of the solution. We refer to Conus et al. \cite{CJK13}
for the asymptotic growth of the solution to stochastic heat equation with Lipschitz continuous coefficient, where a localization method was introduced to study the lower bound of the  asymptotic maximum of the solution. The localization in  \cite{CJK13} is based on Picard iteration of the solution and does not apply to our case. We will make use of associativity to approach the lower bound. We impose the following condition on the distribution of $u(t,0)$. 

\noindent Assumption 1.
There exist $\alpha_0, K_0, C_0>0$ such that for all $z\geq K_0$ and $\delta>0$
\begin{align}\label{eq:condition}
\P\left\{u(t,0) \in [z, z+\delta]\right\} \leq C_0 \delta^{\alpha_0}.
\end{align}

Under the above assumption, we have the following estimate on the joint probability of the solution to \eqref{eq:SHE}.
\begin{lemma} 
            Let $K_0$ be the constant in  Assumption 1.
           There exists a constant $C>0$ such that for all $x,y \in \R$
           \begin{align}\label{eq:joint}
           &\sup_{a, b\geq K_0}\left[ \P\left\{u(t,x) \leq a; u(t,y)\leq b\right\}-  
           \P\left\{u(t,x) \leq a\right\} \P\left\{ u(t,y)\leq b\right\}\right] \nonumber \\
           &
           \qquad\qquad \qquad \qquad \qquad \qquad
           \leq C\cdot\Cov(u(t,x)\,, u(t,y))^{\alpha_0/(2+\alpha_0)}.
           \end{align}
\end{lemma}
\begin{proof}
           The proof is a modification of that of \cite[Theorem 6.2.15]{PR12}.  
                      For $a, \delta>0$, we define the function
           \begin{align*}
           h_{a, \delta}(z)=
           \begin{cases}
           1,& \text{if $z \leq a$}\\
           \frac{a+\delta- z}{\delta},& \text{if $z\in [a, a+\delta]$}\\
           0, &\text{if $z>a+\delta$}.
           \end{cases}
           \end{align*}
           For $a, b\geq K_0$, we write
           \begin{align*}
           &\P\left\{u(t,x) \leq a; u(t,y)\leq b\right\}-  
           \P\left\{u(t,x) \leq a\right\} \P\left\{ u(t,y)\leq b\right\}\\
           &\qquad =\Cov(\bm{1}_{\{u(t,x)\leq a\}}\,, \bm{1}_{\{u(t,y)\leq b\}})\\
           &\qquad = \Cov(h_{a, \delta}(u(t,x))\,, h_{b, \delta}(u(t,y))) \\
           &\qquad\quad + 
           \Cov(\bm{1}_{\{u(t,x)\leq a\}} - h_{a, \delta}(u(t,x))\,, \bm{1}_{\{u(t,y)\leq b\}} )\\
           &\qquad \quad +            \Cov(h_{a, \delta}(u(t,x)) \,, \bm{1}_{\{u(t,y)\leq b\}}-h_{b, \delta}(u(t,y)) ).
           \end{align*}
           Because $u$ is associated, by the inequality \eqref{Newman2},
           \begin{align*}
            \Cov(h_{a, \delta}(u(t,x))\,, h_{b, \delta}(u(t,y))) & \leq \lip((h_{a, \delta})\lip(h_{b, \delta}) \Cov(u(t,x)\,, u(t,y))\\
            &=\delta^{-2}\Cov(u(t,x)\,, u(t,y)). 
           \end{align*}
           Since $|\bm{1}_{\{u(t,x)\leq a\}} - h_{a, \delta}(u(t,x))| \leq \bm{1}_{\{a\leq u(t,x)\leq a+\delta\}}$, we have
           \begin{align*}
           |\Cov(\bm{1}_{\{u(t,x)\leq a\}} - h_{a, \delta}(u(t,x))\,, \bm{1}_{\{u(t,y)\leq b\}} )| \leq 2\P\{u(t,0)\in [a, a+\delta]\}.
           \end{align*}
           Similarly, 
           \begin{align*}
           |\Cov(h_{a, \delta}(u(t,x)) \,, \bm{1}_{\{u(t,y)\leq b\}}-h_{b, \delta}(u(t,y)) )| \leq 2\P\{u(t,0)\in [b, b+\delta]\}.
           \end{align*}
           The preceding yields for all $\delta>0$
           \begin{align*}
           &\P\left\{u(t,x) \leq a; u(t,y)\leq b\right\}-  
           \P\left\{u(t,x) \leq a\right\} \P\left\{ u(t,y)\leq b\right\}\\
           &\qquad\qquad\qquad \leq \delta^{-2}\Cov(u(t,x)\,, u(t,y)) + 2\P\{u(t,0)\in [a, a+\delta]\}\\
           &\qquad\qquad\qquad\quad +2\P\{u(t,0)\in [b, b+\delta]\}\\
           &\qquad\qquad\qquad \leq  \delta^{-2}\Cov(u(t,x)\,, u(t,y)) + 4C_0\delta^{\alpha_0},
           \end{align*}
           where the last inequality is due to \eqref{eq:condition}.  We choose
           \begin{align*}
           \delta = \Cov(u(t,x)\,, u(t, y))^{1/(2+\alpha_0)}
           \end{align*}
           to obtain \eqref{eq:joint}.
\end{proof}

\begin{proposition}
           Almost surely,
          \begin{align}\label{eq:liminf}
        \liminf_{N\to\infty} \frac{\max_{0\leq x\leq N}u(t,x)}{\log N} \geq c_1 t^{1/2},
     \end{align}
     where $c_1>0$ is the constant in \eqref{eq:tail}.
\end{proposition}
\begin{proof} 
            As in the proof of \cite[Theorem 1.1]{CJK13}, we will use Borel-Cantelli lemma to prove \eqref{eq:liminf}.
           Choose $\tilde{c}\in (0, t^{1/2}/c_1)$.  Let $a\in (0, 1)$ such that
           $a>\tilde{c}c_1 t^{-1/2}$. Define $x_j= j N/\floor{N^a}$ for $j= 1, \ldots, \floor{N^a}$.
           Assume that $N$ is sufficiently large so that $\tilde{c}\log N$ is larger than $K_0$, where $K_0$ is the constant in 
           Assumption 1.
           We write
           \begin{align*}
           &\P\left\{ \max_{0\leq x\leq N} u(t,x)  \leq \tilde{c} \log N\right\} \leq           
            \P\left\{ \max_{1\leq j\leq \floor{N^a}} u(t,x_j)  \leq \tilde{c} \log N\right\}\\
            &\quad =            \P\left\{ \max_{1\leq j\leq \floor{N^a}} u(t,x_j)  \leq \tilde{c} \log N\right\}
            -\prod_{j=1}^{\floor{N^a}}\P\left\{u(t,x_j)  \leq \tilde{c} \log N\right\}\\
            &\qquad + \left(1-\P\left\{u(t,0)  > \tilde{c} \log N\right\}\right)^{\floor{N^a}}\\
            &\quad :=\mathcal{P}_{N,1} + \mathcal{P}_{N,2}.
           \end{align*}
           Using \eqref{eq:tail} and the inequality $1-z \leq \e^{-z}$ for all $z\geq0$, we see that
           \begin{align*}
           \mathcal{P}_{N,2} &\leq \left(1- \e^{-\tilde{c}c_1t^{-1/2}\log N}\right)^{\floor{N^a}} = \left(1- N^{-\tilde{c}c_1t^{-1/2}}\right)^{\floor{N^a}}\\
           &\leq \exp\left(-\floor{N^a}N^{-\tilde{c}c_1t^{-1/2}}\right).
           \end{align*}
           Since $\tilde{c}c_1 t^{-1/2}<a$, it follows that
           \begin{align*}
           \sum_{N=1}^\infty \mathcal{P}_{N,2} <\infty.
           \end{align*}
           
           We proceed to estimate $\mathcal{P}_{N,1}$. Since $u$ is 
associated, by \cite[(11)]{Pu25} we have
           \begin{align*}
           \mathcal{P}_{N,1}& \leq \sum_{1\leq j<k\leq \floor{N^a}} \big(\P\left\{u(t, x_j)\leq \tilde{c}\log N; 
           u(t, x_k)\leq \tilde{c}\log N\right\} \\
           &\qquad \qquad \qquad \quad- \P\left\{u(t, x_j)\leq \tilde{c}\log N\right\}\P\left\{
           u(t, x_k)\leq \tilde{c}\log N\right\}\big)\\
           & \leq C\sum_{1\leq j<k\leq \floor{N^a}} \Cov(u(t, x_j)\,, u(t, x_k))^{\alpha_0/(2+\alpha_0)},
           \end{align*}
           where the second inequality holds by \eqref{eq:joint}. Because $|x_j-x_k|\geq N^{1-a}$ for $j< k$, by Lemma \ref{lem:cov}, 
           \begin{align*}
           \Cov(u(t, x_j)\,, u(t, x_k)) \lesssim \e^{-\frac{N^{2(1-a)}}{4T}}.
           \end{align*}
           Thus, 
           \begin{align*}
           \mathcal{P}_{N,1} \lesssim  \sum_{1\leq j<k\leq \floor{N^a}}  \e^{-\frac{\alpha_0}{2+\alpha_0}\frac{N^{2(1-a)}}{4T}}
           \leq N^{2a}  \e^{-\frac{\alpha_0}{2+\alpha_0}\frac{N^{2(1-a)}}{4T}}
           \end{align*}
           and hence we have
           \begin{align*}
           \sum_{N=1}^{\infty}\mathcal{P}_{N,1} <\infty.
           \end{align*}
           Therefore, we obtain 
           \begin{align*}
           \sum_{N=1}^\infty \P\left\{ \max_{0\leq x\leq N} u(t,x)  \leq \tilde{c} \log N\right\}  <\infty.
           \end{align*}
           By Borel-Cantelli lemma, it follows that almost surely
           \begin{align}\label{eq:inf}
           \liminf_{N\to\infty}\frac{ \max_{0\leq x\leq N} u(t,x)}{\log N} \geq \tilde{c}.
           \end{align}
          The limit in \eqref{eq:inf} is taken along integer $N$. Because the quantity $\max_{0\leq x\leq N} u(t,x)$ is 
          increasing in $N$, the limit in \eqref{eq:inf} also holds along real numbers $N$.
           Finally, we let $\tilde{c}$ increase to $c_1t^{1/2}$ to obtain \eqref{eq:liminf}.
\end{proof}

\begin{appendix}
\section{}
The following lemma might be available somewhere in the literature.  We include the proof for reader's convenience.

\begin{lemma}\label{Gauss-linear}
          If for all $a_1, \ldots, a_k\in \R_+$, the linear 
          combination $a_1X_1+\ldots + a_kX_k$ is a centered Gaussian random variable, then $(X_1, \ldots, X_k)$ is a centered 
          Gaussian random
          vector. 
\end{lemma}
\begin{proof}
          The condition ensures that $X_j$ is a centered Gaussian random variable for $j=1, \ldots, k$.
          We need show that for all $a_1, \ldots, a_k\in \R$, the linear 
          combination $a_1X_1+\ldots + a_kX_k$ is a Gaussian random variable. 
          Fix $a_2, \ldots, a_k \in \R_+$  and $\theta\in \R$. Define
          \begin{align*}
          F_1(z)= \E\left[\e^{{\rm i}\theta (zX_1+a_2X_2+\ldots+ a_kX_k)}\right],\quad z\in \C.
          \end{align*}
          For $z=x+\i y$, $x,y\in\R$, since $X_1$ is a centered Gaussian random variable and
          \begin{align*}
          \left| \e^{{\rm i}\theta (zX_1+a_2X_2+\ldots+a_kX_k)} \right| =\e^{-\theta yX_1}
          \end{align*}
          we see that $F_1(z)$ is well defined for all $z\in \C$. Define
          \begin{align*}
          G_1(z)= \e^{-\frac12\theta^2\Var(zX_1+a_2X_2+\ldots+ a_kX_k)}, \quad z\in \C.
          \end{align*}
          It is clear that $G_1$ is analytic on $\C$ and the condition ensures that for all $\theta\in \R$
          \begin{align}\label{ID}
          F_1(a)=G_1(a), \quad \text{for all $a\geq0$}. 
          \end{align}
          We proceed to show that for each fixed $\theta\in \R$, $F_1$ is 
analytic on $\C$. Fix $z_0=x_0+{\rm i} y_0$. 
          Then for $0<|z-z_0|<1$,
          \begin{align*}
          &\frac{F_1(z)-F_1(z_0)}{z-z_0} - \E\left[ {\rm i} \theta X_1\e^{ {\rm i}  \theta(z_0X_1+a_2X_2+\ldots +a_kX_k)} \right] \\
          &\quad
          = \frac{1}{z-z_0}
          \E\Big[ \e^{{\rm i}\theta (zX_1+a_2X_2+\ldots+a_kX_k)} - \e^{{\rm i}\theta (z_0X_1+a_2X_2+\ldots+a_kX_k)}  \\
          &\qquad\qquad\qquad\qquad\qquad\quad
          -(z-z_0) {\rm i}\theta X_1\e^{{\rm i} \theta(z_0X_1+a_2X_2+\ldots +a_kX_k)}
          \Big]\\
          &\quad=  \frac{\E\left[ \e^{{\rm i}\theta (z_0X_1+a_2X_2+\ldots+a_kX_k)}  \left( \e^{{\rm i}\theta (z-z_0)X_1} -1
          -(z-z_0)  {\rm i} \theta X_1\right)
          \right]}{z-z_0}.
          \end{align*}
          Using the Taylor series, we have
          \begin{align*}
            \frac{ \e^{{\rm i}\theta (z-z_0)X_1} -1
          -(z-z_0)  {\rm i}  \theta X_1}{z-z_0}= \sum_{n=2}^\infty \frac{( {\rm i}  \theta X_1)^n(z-z_0)^{n-1}}{n!}.          
          \end{align*}
          Therefore, 
          \begin{align*}
          &\left|  \frac{ \e^{{\rm i}\theta (z_0X_1+a_2X_2+\ldots+a_kX_k)}  \left( \e^{{\rm i}\theta (z-z_0)X_1} -1
          -(z-z_0)  {\rm i}  \theta X_1\right)
          }{z-z_0}  \right|\\
          &\qquad\qquad \leq \e^{-y_0X_1}|\theta X_1| \sum_{n=0}^\infty \frac{|\theta X_1|^n|z-z_0|^{n}}{n!}
          =  \e^{-y_0X_1}|\theta X_1| \e^{|\theta| |z-z_0||X_1|}\\
          &\qquad\qquad \leq \e^{-y_0X_1}|\theta X_1| \e^{|\theta||X_1|}.
          \end{align*}
          Since $X_1$ is a Gaussian random variable, the random variable in the above line has finite expectation. Hence, by dominated
          convergence theorem, 
          \begin{align*}
          \lim_{z\to z_0}\frac{F_1(z)-F_1(z_0)}{z-z_0} = \E\left[  {\rm i}  \theta X_1\e^{ {\rm i}  \theta(z_0X_1+a_2X_2+\ldots+a_kX_k)} \right],
          \end{align*}
          which implies that $F$ is is analytic on $\C$. Hence, by \eqref{ID} and identity principle,  for all fixed $a_2, \ldots, a_k \in \R_+$ and $\theta\in \R$
          \begin{align*}
          F_1(z)= G_1(z), \quad \text{for all $z\in \C$}.
          \end{align*}
          In particular, we have for all $a_1, \theta\in \R$ and $a_2, \ldots, a_k\in \R_+$,
          \begin{align*}
           \E\left[\e^{{\rm i}\theta (a_1X_1+a_2X_2+\ldots +a_kX_k)}\right] = \e^{-\frac12\theta^2\Var(a_1X_1+a_2X_2+\ldots+a_kX_k)}.
          \end{align*}
          
          Similarly, we fix $a_1, \theta\in \R$ and $a_3, \ldots, a_k \in \R_+$ and define
          \begin{align*}
          F_2(z)=  \E\left[\e^{{\rm i}\theta (a_1X_1+zX_2+\ldots +a_kX_k)}\right],\quad z\in \C.
          \end{align*}
          By the same arguments as before, we can show that for all $a_1, a_2, \theta\in \R$ and $a_3, \ldots, a_k\in \R_+$, 
          \begin{align}\label{eq:Gauss}
          \E\left[\e^{{\rm i}\theta (a_1X_1+a_2X_2+\ldots +a_kX_k)}\right] = \e^{-\frac12\theta^2\Var(a_1X_1+a_2X_2+\ldots+a_kX_k)}.
          \end{align}
          We can repeat to conclude that \eqref{eq:Gauss} holds for all $a_1, \ldots, a_k\in \R$ and $\theta$. Therefore, $(X_1, \ldots, X_k)$
          is a centered Gaussian random vector.         
\end{proof}

\end{appendix}

\noindent\textbf{Acknowledgement}.  
FP was supported in part by  National Natural Science Foundation of China (No.12571153).
This material is based upon work supported by the National Science
Foundation under Grant No. DMS-2424139, while the authors were in
residence at the Simons Laufer Mathematical Sciences Institute in
Berkeley, California, during the Fall 2025 semester.


\begin{thebibliography}{999}

\bibitem{BaZ24}
Balan, M. R. and Zheng, G. (2024).
Hyperbolic Anderson model with L\'{e}vy white noise: spatial ergodicity and fluctuation.
{\it Trans. Amer. Math. Soc.} {\bf 377} 4171--4221.


\bibitem{BNZ21}
Bolaños Guerrero, R., Nualart, D. and Zheng, G. (2021). Averaging 2d stochastic wave equation. {\it Electron. J. Probab.} {\bf 26},  pp.1-32.

\bibitem{BuS98}
Bulinskiĭ, A. V. and Shabanovich, \`E. (1998).
Asymptotic behavior of some functionals of positively and negatively dependent random fields. {\it Fundam. Prikl. Mat.}
{\bf 4} 479--492.


\bibitem{CKNP21}
Chen, L., Khoshnevisan, D., Nualart, D. and Pu, F. (2021)
Spatial ergodicity for SPDEs via Poincar\'{e}-type inequalities.
{\it Electron. J. Probab.} {\bf26}  Paper No. 140, 37 pp.

\bibitem{CKNP23}
Chen, L., Khoshnevisan, D., Nualart, D. and Pu, F. (2023).
Central limit theorems for spatial averages of the stochastic heat equation via Malliavin-Stein's method.
{\it Stoch. Partial Differ. Equ. Anal. Comput.} {\bf11}  no. 1, 122–176.

\bibitem{CX23}
Chen, L. and Xia, P.:
Asymptotic properties of stochastic partial differential equations in the sublinear regime. arXiv:2306.06761 (2023)


\bibitem{CJK13}
Conus, D.,  Joseph, M. and Khoshnevisan, D. (2013)
On the chaotic character of the stochastic heat equation, before the onset of intermitttency.
{\it Ann. Probab.} {\bf 41}  no. 3B, 2225–2260.

\bibitem{Dal99}
Dalang, R.C. (1999)
Extending the martingale measure stochastic integral with applications to spatially homogeneous s.p.d.e.'s. 
{\it Electron. J. Probab.} {\bf 4}  no. 6, 29 pp.

\bibitem{DNZ20}
Delgado-Vences, F., Nualart, D., and Zheng, G. (2020). A central limit theorem for the stochastic wave equation with fractional noise. {\it Ann. Inst. Henri Poincar\'e Probab. Stat.}  {\bf 56}  no. 4, 3020–3042.3020-3042.

\bibitem{Ebi24}
Ebina, M. (2024).
Central limit theorems for nonlinear stochastic wave equations in dimension three.
{\it Stoch. Partial Differ. Equ. Anal. Comput.} {\bf 12} no. 2, 1141–1200.

\bibitem{Ebi25}
Ebina, M. (2025).
Central limit theorems for stochastic wave equations in high dimensions.
{\it Electron. J. Probab.} {\bf 30} Paper No. 60, 56 pp.

\bibitem{Ebi25b}
Ebina, M. (2024). Large deviations for a spatial average of stochastic heat and wave equations.
arXiv:2409.15624

\bibitem{EPW67}
	Esary, J. D., Proschan, F. and Walkup, D. W. (1967).
	Association of random variables with applications.
	{\it Ann.\ Math.\ Statist.}\ {\bf 38} 1466--1474. 



\bibitem{HWXZ24}
Hu, Y., Wang, X., Xia, P. and Zheng, J. (2024)
Moment asymptotics for super-Brownian motions.
{\it Bernoulli} {\bf30}  no. 4, 3119–3136.


\bibitem{HNV20}
Huang, J.,  Nualart, D. and Viitasaari, L. (2020)
 A central limit theorem for the stochastic heat equation.
{\it Stochastic Process. Appl.} {\bf 130}  no. 12, 7170–7184.

	\bibitem{HNVZ20}
	Huang, J., Nualart, D., Viitasaari, L., Zheng, G. (2020). Gaussian fluctuations for the stochastic heat equation with colored noise.
{\it Stoch. Partial Differ. Equ. Anal. Comput.} {\bf8} no. 2, 402–421.


\bibitem{LP23} Li, Z. and Pu, F. (2023) Gaussian fluctuation for spatial average of super-Brownian motion. 
 \textit{Stoch. Anal. Appl.} {\bf 41} no.4, 752--769 
 
 \bibitem{Mue91}
 Mueller, C. (1991).
On the support of solutions to the heat equation with noise.
{\it Stochastics Stochastics Rep}. {\bf 37} no. 4, 225–245.

\bibitem{MP92}
Mueller, C. and Perkins, E. A. (1992)
The compact support property for solutions to the heat equation with noise.
{\it Probab. Theory Related Fields} {\bf93} no. 3, 325–358.

\bibitem{Myt98}
Mytnik, L. (1998)
Weak uniqueness for the heat equation with noise.
{\it Ann. Probab.} {\bf 26} no. 3, 968–984.


\bibitem{MP11}
Mytnik, L. and Perkins, E. (2011)
Pathwise uniqueness for stochastic heat equations with H\"older continuous coefficients: the white noise case.
{\it Probab. Theory Related Fields} {\bf 149} no. 1-2, 1–96.

\bibitem{MPS06}
Mytnik, L.,  Perkins, E. and Sturm, A. (2006).
On pathwise uniqueness for stochastic heat equations with non-Lipschitz coefficients.
{\it Ann. Probab.} {\bf34} no. 5, 1910–1959


	\bibitem{NXZ22}
	Nualart, D., Xia, P.,  Zheng, G. (2022). Quantitative central limit theorems for the parabolic Anderson model driven by colored noises. {\it Electron. J. Probab.} {\bf 27}  Paper No. 120, 43 pp.

\bibitem{NZ20}
Nualart, D. and Zheng, G. (2020).
Spatial ergodicity of stochastic wave equations in dimensions $1, 2$ and $3$.
{\it Electron. Commun. Probab.} {\bf 25} Paper No. 80, 11 pp.

\bibitem{New80}
Newman, C. M. (1980).
Normal fluctuations and the FKG inequalities.
{\it Comm. Math. Phys.} {\bf 74}  no. 2, 119–128.


\bibitem{Oli12}
Oliveira, P.E. (2012).
{\it Asymptotics for associated random variables}.
Springer, Heidelberg. x+194 pp.


\bibitem{Per02}
Perkins, E. (2002).
{\it Dawson-Watanabe superprocesses and measure-valued diffusions}. Lectures on probability theory and statistics (Saint-Flour, 1999), 125–324.
Lecture Notes in Math., 1781
Springer-Verlag, Berlin

\bibitem{PR12}
Prakasa Rao, B. L. S. (2012).
{\it Associated Sequences, Demimartingales and Nonparametric inference}.
Probability and its Applications. Birkh\"auser/Springer, Basel.


\bibitem{Pu25}
Pu, F. (2025). Ergodicity, CLT and asymptotic maximum of the Airy1 process. 
{\it Bernoulli} {\bf 31}  2624-2648


\bibitem{Shi94}
Shiga, T. (1994).
Two contrasting properties of solutions for one-dimensional stochastic partial differential equations.
{\it Canad. J. Math.} {\bf 46} no. 2, 415–437.


\bibitem{Wal86}
	Walsh, J. B. (1986).
	{\it An Introduction to Stochastic Partial Differential Equations.}
	\`Ecole d'\'et\'e de probabilit\'es de Saint-Flour, XIV-1984, 265--439.
	In: {\it Lecture Notes in Math.}\ {\bf 1180}, Springer, Berlin.

 \end{thebibliography}
\end{document}